\documentclass[11pt,a4paper]{article}
\usepackage{graphicx}
\usepackage[utf8]{inputenc}
\usepackage[T1]{fontenc}
\usepackage{amsmath}
\usepackage{amsthm}
\usepackage{amsfonts}
\usepackage{amssymb}
\usepackage{enumitem}

\usepackage{natbib}

\usepackage{url}
\usepackage{slashed}
\usepackage{bbm}
\usepackage{tikz-cd}
\usepackage{mathtools}

\usepackage{ifthen}
\usepackage{xspace}
\usepackage{fancyhdr}
\usepackage{aliascnt}
\usepackage[textsize=small]{todonotes}
\PassOptionsToPackage{pagebackref}{hyperref}
\usepackage{bookmark}
\usepackage{caption}
\usepackage{leftidx}

\usetikzlibrary{calc}
\usepackage{cancel}

\usepackage[margin=1in]{geometry}
\usepackage[T1]{fontenc}
\usepackage{ae,aecompl}
\usepackage{bm}

\usepackage{microtype}

\usepackage{setspace}
\tikzset{curve/.style={settings={#1},to path={(\tikztostart)
    .. controls ($(\tikztostart)!\pv{pos}!(\tikztotarget)!\pv{height}!270:(\tikztotarget)$)
    and ($(\tikztostart)!1-\pv{pos}!(\tikztotarget)!\pv{height}!270:(\tikztotarget)$)
    .. (\tikztotarget)\tikztonodes}},
    settings/.code={\tikzset{quiver/.cd,#1}
        \def\pv##1{\pgfkeysvalueof{/tikz/quiver/##1}}},
    quiver/.cd,pos/.initial=0.35,height/.initial=0}

\newtheorem{theorem}{Theorem}[section]
\newaliascnt{prop}{theorem}
\newtheorem{prop}[prop]{Proposition}
\aliascntresetthe{prop}
\newaliascnt{lemma}{theorem}
\newtheorem{lemma}[lemma]{Lemma}
\aliascntresetthe{lemma}
\newaliascnt{cor}{theorem}
\newtheorem{cor}[cor]{Corollary}
\aliascntresetthe{cor}
\newaliascnt{con}{theorem}

\aliascntresetthe{con}
\newaliascnt{assumption}{theorem}

\aliascntresetthe{assumption}
\newaliascnt{hyp}{theorem}

\aliascntresetthe{hyp}

\theoremstyle{definition}
\newaliascnt{definition}{theorem}
\newtheorem{definition}[definition]{Definition}
\aliascntresetthe{definition}
\newaliascnt{example}{theorem}

\aliascntresetthe{example}
\newaliascnt{convention}{theorem}

\aliascntresetthe{convention}
\newaliascnt{rem}{theorem}
\newtheorem{rem}[rem]{Remark}
\aliascntresetthe{rem}
\newaliascnt{fact}{theorem}

\aliascntresetthe{fact}

\newtheorem{ingredient}{Theorem}

\newcommand{\loc}{\textnormal{loc}}
\newcommand{\sing}{\textnormal{sing}}
\newcommand{\supp}{\textnormal{supp }}
\newcommand{\End}{\textnormal{End}}

\newcommand{\pr}{\textnormal{pr}}

\newcommand{\im}{\textnormal{im\,}}

\newcommand{\vol}{\textnormal{vol}}

\newcommand{\id}{\textnormal{Id}}
\newcommand{\coker}{\textnormal{coker\,}}

\newcommand{\Conf}{\textnormal{Conf}}
\newcommand{\Rat}{\textnormal{Rat}}

\DeclareMathOperator{\ind}{ind}

\newcommand{\eand}{\quad \text{and} \quad}

\newcommand{\fF}{\mathfrak{F}}

\newcommand{\fp}{\mathfrak{p}}

\newcommand{\cE}{\mathcal{E}}
\newcommand{\cC}{\mathcal{C}}
\newcommand{\cD}{\mathcal{D}}
\newcommand{\cH}{\mathcal{H}}
\newcommand{\cF}{\mathcal{F}}
\newcommand{\cI}{\mathcal{I}}
\newcommand{\cP}{\mathcal{P}}
\newcommand{\cR}{\mathcal{R}}
\newcommand{\cM}{\mathcal{M}}
\newcommand{\cN}{\mathcal{N}}

\newcommand{\cO}{\mathcal{O}}
\newcommand{\cB}{\mathcal{B}}

\newcommand{\cV}{\mathcal{V}}

\newcommand{\cZ}{\mathcal{Z}}

\newcommand{\bN}{\mathbb{N}}

\newcommand{\bR}{\mathbb{R}}
\newcommand{\bC}{\mathbb{C}}

\newcommand{\bP}{\mathbb{P}}

\newcommand{\bx}{\mathbf{x}}

\newcommand{\bev}{\mathbf{ev}}
\newcommand{\ev}{\mathrm{ev}}

\newcommand{\Comment}[2][\empty]{\ifthenelse{\equal{#1}{\empty}}{\todo[color=gray!10]{#2}\ }{\todo[color=gray!10,#1]{#2}}}

\usepackage{hyperref}

\makeatletter
\let\originalHyRef@testreftype\HyRef@testreftype
\def\HyRef@testreftype#1.#2\\{%
  \def\my@reftype{#1}%
  \def\my@itemtype{Item}%
  \ifx\my@reftype\my@itemtype
    \def\HyRef@currentHtag{}%
  \else
    \originalHyRef@testreftype#1.#2\\%
  \fi
}
\makeatother

\renewcommand*{\backrefalt}[4]{%
  \ifcase #1\relax
  \or (Cit. on p.~#2)%
  \else (Cit. on pp.~#2)%
  \fi
}

\title{On bubbling of Fueter maps}
\author{Jacek Rzemieniecki\thanks{Humboldt Universit\"at zu Berlin. Email: \texttt{jacek.rzemieniecki[]hu-berlin.de}}
}
\date{October 27, 2026}

\begin{document}

\maketitle\vspace{-2em}
\begin{abstract}

We study bubbling for sequences of solutions of the Hamiltonian-perturbed Fueter equation. By Walpuski's compactness theorem, a bounded-energy sequence converges weakly to a Fueter map, with energy loss recorded by a defect measure supported on a codimension $2$ rectifiable set. Assuming that the limiting map has no non-removable singularities and the bubbles attach to it on a subset of the bubbling locus of positive measure, we show that the limiting Hamiltonian lies in an exceptional subset of infinite codimension. The proof combines quantitative control of rational sweepout loci in hyperkähler manifolds with a transversality argument. We expect the attachment assumption to be automatic and follow from a suitable no-neck result. This strongly suggests that, in the absence of non-removable singularities, bubbling of Hamiltonian-perturbed Fueter maps is highly nongeneric.
\end{abstract}

\section{Introduction}

Let $(M, g_M)$ be a closed oriented Riemannian $3$--manifold and let $(X, g_X, I_1, I_2, I_3)$ be a hyperk\"ahler manifold. Fix a fiberwise isometric identification $\cI: STM \to S(X)$, where $S(X)\coloneq \{ a_1 I_1+ a_2 I_2 + a_3 I_3 \, | \, (a_1, a_2, a_3) \in S^2\}$ denotes the twistor sphere of $X$. Consider the following parameter space $\cP \coloneq C^\infty(M \times X, \bR)$. These data, together with a parameter $H \in \cP$ define a \textbf{Fueter operator} $\cF_{H}$ acting on functions $f\in C^\infty(M, X)$ as follows
\begin{equation}\label{perturbed Fueter equation}
    \cF_{H} (f) \coloneq \sum_{a=1}^3 \cI(e_a) Tf(e_a) - \nabla^{g_X} H(f) \in \Gamma(f^*TX),
\end{equation} 
where $\{e_a\}$ is an orthonormal basis of $TM$ at a point and $\nabla^{g_X} H$ denotes the gradient of $H$ in the direction of $X$. Notice that this does not depend on the choice of orthonormal basis. The solutions $f$ of $\cF_{H} (f) = 0$ are called \textbf{Fueter maps}. Sometimes we will write $H$--Fueter maps to stress the dependence on the parameter.

The Fueter equation in the modern formulation was originally introduced by Taubes in \cite{Ta99} and it appears naturally in the study of gauge theory in higher dimensions \cite{DS11}, \cite{Ha12}, \cite{HL24}, \cite{Wa17a}, \cite{Wa17b} and generalized Seiberg--Witten equations \cite{Do19}, \cite{DW20}, \cite{Ha15}, \cite{HW15}, \cite{Wi26} and plays a prominent role in the recent proposal of Doan--Rezchikov \cite{DR22} to categorify invariants from symplectic topology. 

One of the early motivations for studying Fueter maps is the work of Hohloch--Noetzel--Salamon \cite{HNS09}, \cite{Sa13} who introduce an action functional whose critical points are Hamiltonian-perturbed Fueter maps and develop a Floer theory in the very restricted case of a compact flat target $X$. The main difficulty in extending this theory to more general targets such as, for example, $K3$ surfaces, lies in the possible non-compactness phenomena for sequences of solutions to the Fueter equation.

The following results highlight what is known about the compactness problem for Fueter maps and provide a summary of the works of Walpuski \cite{W15} and Bellettini--Tian \cite{BT19} which should be viewed as analogues to the foundational results in the compactness theory of stationary harmonic maps of Lin \cite{Li99} and Lin--Rivière \cite{LR02}, respectively. Suppose that $X$ is compact and $(H_i)$ is a sequence of Hamiltonians which converges to $H \in \cP$. Let $(f_i)$ be a sequence of solutions of the $H_i$--perturbed Fueter equation satisfying a uniform energy bound $\int_{Y} |T f_i|^2 \leq c_E$. Then, after possibly passing to a subsequence, one can conclude the following:
\begin{itemize}
  \item There is a closed subset $S \subset M$ and a $H$--Fueter map $f \in C^\infty(M \setminus S, X)$ such that $f_i |_{M\setminus S} \to f$ in $C^\infty_\loc$. Moreover, there is a constant $\epsilon_0 > 0$ and an upper semi-continuous function $\Theta: S \to [\epsilon_0,\infty)$ such that $|T f_i|^2 \,\cH^3 \rightharpoonup |T f|^2 \,\cH^3 + \Theta \,\cH^1\lfloor S$ as Radon measures.
  \item The set $S$ decomposes as $S = \Gamma \cup \sing(f)$ into (possibly intersecting) subsets such that $\Gamma \coloneq \supp(\Theta \,\cH^1\lfloor S)$ is $\cH^1$--rectifiable and $\cH^1(\sing(f))=0$. The map $f$ extends smoothly\footnote{This is not stated in the main theorem but explained in \cite[Fact (c), pp. 1296]{BT19}.} away from the locus of nonremovable singularities $\sing(f)$.
  \item For $\cH^1$-almost every $x \in \Gamma$, there exist a positive integer $n_x$ and non-constant holomorphic spheres $u_j : (S^2, i) \to (X,-\cI(v))$ with $v$ a unit vector spanning the approximate tangent line $T_x\Gamma$ and $j=1, \dots, n_x$ such that the energy is quantized over the whole bubbling tree
  \begin{equation}\label{energy quantization}
    \Theta(x) = \sum_{j=1}^{n_x} E(u_j).
  \end{equation}
\end{itemize}

The purpose of this article is to investigate the genericity of such bubbling in the absence of nonremovable singularities. The following is our main result. 
\begin{theorem}\label{main theorem}
  Let $(M, g_M)$ be a closed oriented Riemannian $3$--manifold and let $(X, g_X, I_1, I_2, I_3)$ be a closed hyperk\"ahler manifold. Suppose that $(H_i)$ is a sequence of parameters in $\cP$ converging to $H \in \cP$. Let $(f_i)$ be a sequence of solutions of the $H_i$--perturbed Fueter equation satisfying a uniform energy bound. As explained above, after passing to a subsequence, we can extract a limiting $H$--Fueter map $f$ and a bubbling locus $\Gamma$. Assume that
  \begin{enumerate}[
  label=\textnormal{(\roman*)},
  topsep=0pt,
  partopsep=0pt
  ]
    \item\label{no singularities} $f$ extends smoothly across the whole of $M$; 
    \item\label{map connects to bubbles} For a positive $\cH^1$-measure subset of $\Gamma$, a bubble at $x$ passes through $f(x)$.
  \end{enumerate} 
  Then $H \in \cP^\times$, where $\cP^\times$ is an infinite-codimensional\footnote{By this we mean that for any compact finite-dimensional smooth manifold $K$, the set $\{ \fp \in C^\infty(K, \cP) \, | \, \fp(t) \notin \cP^\times \; \text{ for all } t\in K \}$ is residual in $C^\infty(K, \cP)$ endowed with its standard Fréchet topology.} subset of $\cP$.
\end{theorem}

\begin{rem}
  Most of our arguments are local in nature and our proof should adapt to the setup of Fueter sections with perturbations as defined by Walpuski \cite[pp. 753]{W15}. We opted to work with Fueter maps perturbed by a Hamiltonian with the application \cite{HNS09} in mind and in order to streamline the exposition by avoiding the bookkeeping of the varying fiber data. Furthermore, we expect that a slightly more careful use of the Barlet space in Section \ref{Rational curves in hyperkähler manifolds} would allow us to relax the closedness condition for $X$ and instead consider a sequence of $f_i$ whose images lie in a compact subset of $X$. 
\end{rem}

The proof idea can be explained as follows: holomorphic bubbles can occur only in a countable number of directions on the twistor sphere. Moreover, their images (the sweepout loci) are covered by countably many real codimension $2$ submanifolds of $X$. If the bubbles attach to the limiting map $f$ along a set of positive $\cH^1$-measure, by a countability argument, we can find a fixed direction $v$ and a fixed submanifold $Y \subset X$ of codimension $2$ such that $f(x) \in Y$ and $T_x f(v) \in T_{f(x)} Y$ on a set of positive $\cH^1$-measure. Imposing these conditions at $n$ distinct points gives a Fredholm problem of index $-n$: each point contributes $4$ incidence conditions, but only $3$ degrees of freedom. Since $n$ is arbitrary, the exceptional set of Hamiltonians has infinite codimension. 

\autoref{Proof of the main theorem} explains how to make this argument rigorous. This requires three black boxes labelled \autoref{thm:A}, \autoref{thm:B} and \autoref{thm:C} which are then proved in the remaining sections. 

Finally, the following remarks are in order regarding the assumptions in \autoref{main theorem}.

\begin{rem}
  Little is known about singularity formation for Fueter maps. In particular, the following model problem is of fundamental importance: is there a Fueter map from a punctured Euclidean ball $f: B_1(0) \setminus \{0\} \to X$ whose tangent map at $0$ is pure bubbling? In other words, could it happen that there is a sequence of rescalings $f_{r_l}: B_{1/r_l}(0) \setminus \{0\} \to X$ of $f$ with $r_l\searrow 0$ such that $f_{r_l}$ converge in $C^\infty_{\loc}$ to a constant map but as measures $|Tf_{r_l}|^2 \cH^3 \rightharpoonup \nu$, where $\nu$ is a nontrivial dilation--invariant measure supported on a union of finitely many rays through the origin? A negative answer would rule out pure bubbling as a mechanism for singularity formation in limits of Fueter maps.
  
  It is important to point out that similar phenomena occur in other problems in geometric analysis: notably, Chen--Sun \cite[Corollary 1.4]{ChenSun21Analytic} show that there exists an admissible Hermitian Yang--Mills connection on a rank-$2$ reflexive sheaf over $\bC P^3$ such that at every singular point, the analytic tangent cone has a trivial flat connection but nonempty bubbling sets. Even if something similar occurs in our context, it could still be true that Fueter maps with point singularities have a good-enough deformation theory to allow a variation of Theorem \ref{main theorem} to hold. 
  
  Lastly, we remark that in the event that $X$ admits no rational curves, the situation is much better and Esfahani \cite{Es23} observes that there are no non-removable singularities. See also Asselle--Brilleslijper \cite{AB26} for a recent compactness result for Fueter equations arising in a Floer-theoretic setting in the absence of rational curves.
\end{rem}

\begin{rem}
  We expect the attachment assumption \autoref{map connects to bubbles} to be automatic. A stronger no-neck statement would assert that at almost every bubbling point $x$ the union of $\{f(x)\}$ with the images of all bubbles arising at $x$ is connected. This is suggested by the fact that, after rescaling in the directions normal to the bubbling locus, the Fueter equation converges to the Cauchy-Riemann equation and the bubbles are holomorphic spheres. Bellettini--Tian's analysis shows that no Dirichlet energy remains on the neck regions, but this alone does not exclude a slow drift in the target along a rescaled cylinder whose length tends to infinity. An additional oscillation or endpoint estimate appears to be needed to upgrade the delicate Lorentz-space analysis in \cite[Section 5]{BT19} to a diameter-collapse result. For the argument below, however, only the weaker attachment property in \autoref{map connects to bubbles} is required.
\end{rem}

\textbf{Statement on the use of AI:} The author used GPT-6 during the proofreeding of this article to identify a number of typos and inaccuracies and to suggest various improvements in the phrasing. GPT-6 also pointed out that \autoref{HK not unirulled} follows from a classical result due to Fujiki \cite{Fujiki1981}. The mathematical ideas and results are entirely the author's own, and the text was written by the author.

\textbf{Acknowledgments:} The author wishes to thank Thomas Walpuski for insightful discussions and careful proofreading and comments on a draft of this article and Aleksander Doan for his insistence that bubbling of Fueter maps is in the end perhaps an infinite codimension phenomenon. I also wish to thank an anonymous referee for a suggestion that led to a strengthening of the main result. This work was supported by the Deutsche Forschungsgemeinschaft (DFG, German Research Foundation) under Germany's Excellence Strategy through the Berlin Mathematics Research Center MATH+ (EXC-2046/2, project ID: 390685689). 

\section{Proof of the main theorem}\label{Proof of the main theorem}

The purpose of this section is to explain that the main theorem is a consequence of the following three ingredients; the rest of the paper will be dedicated to their proof. We need the following notion.

\begin{definition}
    For a complex manifold $(X, I)$, let 
    \begin{equation*}
        \Rat(X, I) \coloneq \{ x \in X \; | \; x\in C \; \text{for some irreducible rational curve } C\} 
    \end{equation*}
    denote the \textbf{locus of rational curves} on $(X, I)$.
\end{definition}

The first result quantifies the sparsity of rational curves in a hyperkähler manifold.

\begin{ingredient}\label{thm:A}
    Let $(X, g_X, I_1, I_2, I_3)$ be a closed hyperkähler manifold. Then, there exists a countable subset $\{ I_j \}_{j \in \bN}$ and a countable collection $\{Y_k\}_{k \in \bN}$ of connected real codimension-two submanifolds of $X$ such that
    \begin{equation*}
       \Rat(X, I)=\emptyset \quad \text{for } I \in S(X) \setminus  \{ I_j \}
    \end{equation*}
    and 
    \begin{equation*}
        \bigcup_{I \in S(X)} \Rat(X, I) \subset \bigcup_{k} Y_k.
    \end{equation*}
\end{ingredient}

To set the notation for the second ingredient in the proof, we need a definition.

\begin{definition}\label{Setup}
    Let $(M, g_M)$ be a closed oriented Riemannian $3$--manifold and suppose that $(X, g_X, I_1, I_2, I_3)$ is a hyperk\"ahler manifold. Choose a fiberwise isometric identification
    \begin{align*}
        \cI: STM \to S(X).
    \end{align*} 
    Now, for $m \geq 2$ consider the space of \textbf{Hamiltonian perturbations}
    \begin{align*}
        \cP^m \coloneq C^m(M \times X, \bR).
    \end{align*}
    For any $H \in \cP^m$, denote by $\nabla^{g_X} H \in C^{m-1}\Gamma(M \times X, \pr_2^* TX)$ the gradient of $H$ in the direction of $X$. Now, this data yields the \textbf{perturbed Fueter operator} acting on maps $f: M \to X$ by
    \begin{align*}
        \cF_H f \coloneq \sum_{a=1}^3 \cI(e_a) Tf(e_a) - \nabla^{g_X} H (f),
    \end{align*}
    where $\{e_a \}$ is an orthonormal frame of $M$ at a point.

    For any positive integer $n$, define the \textbf{ordered configuration space of $n$ points on $M$} by
    \begin{equation*}
        \Conf_n (M) = \{(x_1, \dots, x_n) \in M^n \; | \; x_i \neq x_j \; \textnormal{for } i \neq j\}.
    \end{equation*}
\end{definition}

The second result is a transversality theorem for Fueter maps with incidence conditions.

\begin{ingredient}\label{thm:B}
    Let $(X, g_X, I_1, I_2, I_3)$ be a closed hyperk\"ahler manifold. Fix a real codimension $2$ submanifold $Z \subset X$ and a unit vector field $v \in \Gamma(\mathbb{S}TM)$. For $p>1$ and $m, k$ positive integers satisfying $k-3/p>2$ and $m \geq k+1$ the \textbf{universal space of Fueter maps with $n$ incidence conditions}
    \begin{equation*}
        \cN^m_n(v, Z) \coloneq
        \Bigg\{ (f, \bx, H ) \in W^{k, p}(M, X) \times \Conf_n (M) \times \cP^m \ \Bigg|\
        \begin{array}{l}
            \cF_H f = 0 \, \text{ and for } \bx = (x_1 , \dots, x_n)\\
            f(x_j) \in Z \text{ and } T_{x_j}f(v) \in T_{f(x_j)}Z \\
            \text{for all } j= 1, \dots, n
        \end{array}
        \Bigg\}.
    \end{equation*}
    is a $C^1$-Banach manifold. Moreover, the natural projection map
    \begin{equation*}
        \pi^m: \cN^m_n (v, Z) \to \cP^m: \quad (f, \bx, H ) \mapsto H
    \end{equation*}
    is Fredholm of index $-n$.
\end{ingredient}

The third and last ingredient in the proof makes use of the following notion.

\begin{definition}\label{Banach rectifiability}
    Let $\cB$ be a second-countable Banach manifold and let $n$ be a positive integer. We call a subset $\cZ \subset \cB$ \textbf{Fredholm $C^r$-rectifiable of codimension $n$} if $\cZ$ is contained in a countable union of images of $C^r$-Fredholm maps of index less than or equal than $-n$.
\end{definition}

The reason for introducing Fredholm rectifiability is that we use it to pass to the notion of infinite codimension at the level of the Fréchet space $\cP$. More abstractly, we have the following.

\begin{ingredient}\label{thm:C}
    Let $V$ be a Hausdorff topological vector space, $W$ a Banach space and $j: V \to W$ a continuous linear map with dense image. Let $K$ be a compact smooth manifold of dimension $k$ and let $\cR \subset W$ be Fredholm $C^1$-rectifiable of codimension $n$. If $k<n$, then 
    \begin{equation*}
        \{ f \in C^\infty(K, V) \, | \, (j\circ f)(K) \cap \cR = \emptyset \}
    \end{equation*} 
    is comeager in $C^\infty(K, V)$.
\end{ingredient}

With these three ingredients in place, we can proceed with the proof of the main theorem.

\begin{proof}[Proof of \autoref{main theorem}]

Let $\{ I_j \}_{j \in \bN}$ and $\{ Y_k \}_{k \in \bN}$ denote the sets from \autoref{thm:A}. Put 
\begin{equation*}
    v_j \coloneq \cI^{-1}(I_j) \in \Gamma(STM).
\end{equation*} 

Fix $m$ sufficiently large and let
\begin{equation*}
    j_m: \cP \to \cP^m
\end{equation*}
denote the natural dense inclusion. For any $n \geq 1$ define 
\begin{equation*}
    \cP^{m, \times}_n \coloneq \bigcup_{j,k} \pi^m\big(\cN^m_n (v_j, Y_k) \big) \eand \cP^\times_n \coloneq j^{-1}_m (\cP^{m, \times}_n ).
\end{equation*} 
By \autoref{thm:B}, the set $\cP^{m, \times}_n$ is Fredholm $C^1$-rectifiable of codimension $n$.

We next show that $H \in \cP^\times_n$ for every $n$.

Let $\Gamma_0 \subset \Gamma$ be the set of positive $\cH^1$-measure on which the approximate tangent to $\Gamma$ exists, the bubble description holds and $f(x)$ lies on a bubble at $x$. For every $x \in \Gamma_0$, the bubbles at $x$ are holomorphic with respect to the complex structure determined, up to sign, by the approximate tangent line $T_x \Gamma$. Consequently,
\begin{equation*}
    \Gamma_0 = \bigcup_j E_j,
\end{equation*}
where $E_j$ consists of those points at which the bubbles are $I_j$-holomorphic and thus $T_x \Gamma$ is spanned by $v_j(x)$. By countable additivity, there is a $j_0$ such that 
\begin{equation*}
    \cH^1(E_{j_0}) >0.
\end{equation*}
By assumption \autoref{map connects to bubbles}, for every $x \in E_{j_0}$ the point $f(x)$ lies on one of the bubbles at $x$. Thus 
\begin{equation*}
    f(x) \in \Rat(X, I_{j_0}) \subset \bigcup_k Y_k.
\end{equation*}
Another countability argument produces $k_0$ such that
\begin{equation*}
    E \coloneq E_{j_0} \cap f^{-1}(Y_{k_0})
\end{equation*}
has positive $\cH^1$-measure. 

Since $\Gamma$ is countably $1$-rectifiable, differentiating $f(E) \subset Y_{k_0}$ along the approximate tangent $T_x\Gamma = \langle v_{j_0}(x) \rangle$ at almost every density point $x$ of $E$ gives
\begin{equation*}
    T_x f \big( v_{j_0}\big) \in T_{f(x)}Y_{k_0}.
\end{equation*}
There are therefore infinitely many such points. For any $n$, choose $n$ of them, $x_1, \dots, x_n$. Then
\begin{equation*}
    \big(f, (x_1, \dots, x_n), H \big) \in \cN^m_n(v_{j_0}, Y_{k_0}),
\end{equation*}
and consequently $H \in \cP^\times_n$.

Since $n$ was arbitrary, it follows that
\begin{equation*}
    H \in \cP^\times \coloneq \bigcap_{n\in \bN} \cP^\times_n.
\end{equation*}
Finally, let $K$ be a compact smooth manifold of dimension $d$ and choose $n>d$. By \autoref{thm:C}, a comeager subset of $C^\infty(K, \cP)$ consists of families whose images avoid $\cP^\times_n$ and thus these families also avoid $\cP^\times$. Hence $\cP^\times$ is infinite-codimensional in the sense of \autoref{main theorem}.
\end{proof}
\section{Rational curves in hyperkähler manifolds}\label{Rational curves in hyperkähler manifolds}

The purpose of this section is to prove \autoref{thm:A}. 

\subsection{Rational sweepout loci in hyperkähler manifolds}

Let us begin by recalling a result of Verbitsky. 

\begin{theorem}[{Verbitsky \cite[Theorem 1.2]{Verbitsky2004}}] \label{Verbitsky's theorem}
  Let $(X, I_1, I_2, I_3, g)$ be a (not necessarily compact) hyperkähler manifold. Then the set of $I \in S(X)$ for which there is a rational curve on $(X, I)$ is at most countable. 
\end{theorem}

\begin{proof}[Proof sketch]
  Suppose that $Z$ is a compact real analytic $2$--cycle on $X$. Now, define a function
  \begin{equation*}
    V_{[Z]}: \, S(X) \to \bR: \; V_Z(L) \coloneq \frac{1}{2} \int_Z \omega_L.
  \end{equation*}
  Clearly, this depends linearly on $L = a I_1 +b I_2 + c I_3$ so $V_Z$ is a restriction of a polynomial to $S(X) \cong \bP^1$.  If $Z\subset (X, I)$ is complex analytic, then Wirtinger's inequality implies that $V_Z$ attains its maximum at $I$. Now, for a given homology class $[Z]$, there are either finitely many such maxima $I$ or the map $V_{[Z]}$ is constant. 
  
  In the case that $V_{[Z]}$ is constant, then $V_{[Z]}(L)= \vol_g(Z)$ for all $L$ and from the characterization of equality in Wirtinger's inequality, it follows that $Z$ is $L$--holomorphic for all $L \in S(X)$ and hence trianalytic. However, trianalytic varieties have real dimension divisible by $4$, contradiction.

  Taking the union of all $L$ for all homology classes $[Z]$ for which $V_{[Z]}$ is nonconstant yields the set in the theorem.
\end{proof}

\begin{rem}
  Walpuski \cite[Proposition 8.1]{W15} gives a strengthening of the above result to a \textit{finite} subset of twistor directions under the extra hypothesis of a closed simple hyperk\"ahler manifold with $b_2(X) \geq 6$ and a fixed area bound on the rational curve. None of this is needed for any of the subsequent arguments.
\end{rem}

In the following, \textbf{analytic variety} will always mean an irreducible reduced complex analytic space. 

\begin{definition}
  We say that a compact complex manifold $X$ of dimension $n$ is \textbf{uniruled} if there exists a diagram
  \[\begin{tikzcd}
    U \arrow{rr}{f} \arrow[swap]{d}{\pi} && X \\
    B
  \end{tikzcd}\]
  where $U$ and $B$ are compact analytic varieties, the dimension of $B$ is $n-1$, $\pi$ is a proper holomorphic surjection whose general fiber is isomorphic to $\mathbb{P}^1$ and $f$ is a dominant holomorphic map.
\end{definition}

The following is a classical fact (see Fujiki \cite[Remark following Lemma 2.2]{Fujiki1981}). We include a short proof, working directly with the definition of uniruledness above. 

\begin{prop}\label{HK not unirulled}
  Let $X$ be a compact Kähler manifold with the canonical bundle $K_X$ trivial. Then $X$ is not uniruled.
\end{prop}

\begin{proof}
  Suppose that $X$ is uniruled and we have a diagram as in the definition above. Let $s\in \mathrm{H}^0(X, K_X)$ be nonzero. Since $U$ is compact of dimension $n$, it follows that $f$ is generically finite and consequently, on the regular part\footnote{This is: the Zariski-open part of $U$ on which $\pi$ is a submersion.} $U^\circ$ it defines a natural map
  \begin{equation*}
    J_f \coloneq \Lambda^n (Tf^\vee): \; f^* K_X \to K_{U^\circ}.
  \end{equation*}
  This yields a holomorphic section
  \begin{equation*}
    \tilde{s} \coloneq J_f (f^* s) \in \mathrm{H}^0(U^\circ, K_{U^\circ}).
  \end{equation*}
  Since $f$ is generically finite, $J_f (f^* s)$ is not identically zero on $U^\circ$. Hence its restriction to $C_b = \pi^{-1}(b)$ is not identically zero for general $b \in B^\circ$.
  
  Now, observe that as $\pi|_{C_b}$ is constant, it holds that
  \begin{equation*}
    K_{U^\circ} = K_{U^\circ/B^\circ} \otimes \pi^*K_{B^\circ}\eand  \pi^* K_{B^\circ}\big|_{C_b} \cong \cO_{C_b},
  \end{equation*}
  so that 
  \begin{equation*}
    K_{U^\circ}\big|_{C_b}  \cong \cO(-2),
  \end{equation*}
  and thus
  \begin{equation*}
    \mathrm{res}_{C_b} \tilde{s} \in \mathrm{H}^0(C_b, K_{U^\circ}\big|_{C_b}) \cong \mathrm{H}^0(\mathbb{P}^1, \cO(-2)) = 0.
  \end{equation*}
  This contradicts the fact that $\tilde{s}|_{C_b}$ is not identically zero for general $b \in B^\circ$.
\end{proof}

In particular, compact hyperk\"ahler manifolds are not uniruled. The following proposition relates this property to the sweepout locus of rational curves.

\begin{prop}\label{rational sweepout}
  Let $(X, I)$ be a compact K\"ahler manifold which is not uniruled. Then the sweepout locus of the rational curves $\Rat(X, I)$ is contained in a countable union of proper analytic subsets of $X$.
\end{prop}

\begin{proof}[Proof of \autoref{thm:A}]
  Every proper analytic subset admits a countable stratification by smooth submanifolds of real codimension at least $2$. Since strata of higher codimension can locally be thickened to smooth submanifolds of codimension $2$, the claim follows from \autoref{Verbitsky's theorem} and \autoref{rational sweepout}.
\end{proof}

The remainder of this section will consist of the proof of \autoref{rational sweepout}. First, some complex-analytic preliminaries are in order.

\subsection{General facts about the cycle space}
The purpose of this subsection is to gather some generalities on Barlet's cycle space which are necessary in the proof of \autoref{rational sweepout}.
\begin{definition}
  Let $X$ be a reduced complex space. The \textbf{Barlet space} of compact $n$--cycles is defined to be the following set of formal sums of analytic subsets of $X$
  \begin{align*}
  \mathcal C_n(X)\coloneq \Bigg\{ \sum_{j=1}^{r} m_j Z_j\ \Bigg|\
  \begin{array}{l}
    Z_j \subset X \text{ are distinct irreducible compact analytic} \\
    \text{subsets of pure dimension } n \, \text{with } m_j \in \mathbb Z_{>0}
  \end{array}
  \Bigg\}.
\end{align*}
  If $Z  =\sum_{j=1}^{r} m_j Z_j \in \cC_n(X)$ is an analytic cycle, we denote by 
  \begin{equation*}
    |Z| \coloneq  \bigcup_{j = 1}^r Z_j
  \end{equation*}
  its \textbf{support}. The cycle $Z$ is called \textbf{reduced} if all $m_j = 1$.

  More generally, if $S$ is a set parametrizing a family of cycles $(Z_s)_{s \in S}$ of $n$--cycles in $X$, we define its \textbf{set-theoretic graph} to be 
  \begin{equation*} 
    |\Gamma_S| \coloneq \{(s,x) \in S \times X \mid x \in |Z_s|\}. 
  \end{equation*} 
  In particular, the \textbf{tautological family} over $\cC_n(X)$ is the family whose member over a point $Z \in \cC_n(X)$ is precisely the cycle $Z$ itself.
\end{definition}

For a comprehensive treatment of the cycle space $\cC_n(X)$ we refer the reader to \cite{BarletMagnusson2019CyclesI}, \cite{BarletMagnusson2025CyclesII}. The fundamental result is that this set carries a natural structure of a reduced complex space with the expected universal property for proper analytic families of compact cycles.

\begin{theorem}[{\cite[Theorem 4.6.1]{BarletMagnusson2019CyclesI}}]
  Let $X$ be a complex space and let $n \geq 0$. Then the set $\cC_n(X)$ of compact analytic $n$--cycles in $X$ admits a natural structure of a reduced complex space such that the tautological family of compact $n$--cycles in $X$, parameterized by $\cC_n(X)$, is a proper analytic family. Moreover, this family has the following universal property:
  
  If $S$ is a reduced complex space and $(Z_s)_{s \in S}$ is a proper analytic family of compact $n$--cycles in $X$, then the associated classifying map $f: S \rightarrow \cC_n(X),\, s \mapsto Z_s $ is holomorphic.
\end{theorem}

\begin{rem}
  The statement of this theorem requires an explanation of what a proper analytic family means. At an intuitive level, an analytic family of cycles is a family of cycles which varies holomorphically in adapted local coordinates.

  More precisely, a family of cycles is called \textbf{analytic} if near every parameter value $s_0 \in S$ and in a scale\footnote{An $n$--scale \cite[Definition 4.2.1]{BarletMagnusson2019CyclesI} on a complex space consists of a triple $E=(U, B, j)$ with $U \subset \bC^n$ and $B \subset \bC^p$ being open relatively compact polydiscs and $j$ is a closed embedding of an open subset $X_E$ of $X$ into an open neighborhood $W$ of $\bar{U}\times \bar{B}$.} $E=(U, B, j)$ adapted\footnote{We say that the scale $E$ is adapted \cite[Definition 4.2.2]{BarletMagnusson2019CyclesI} to an $n$--cycle $Z$ if $j^{-1}(\bar{U} \times \partial B) \cap |Z| = \emptyset$.} to the cycle $Z_{s_0}$, the nearby cycles are represented as finite multigraphs of constant degree $k$ over $U$ which are holomorphic maps
  \begin{equation*}
    f_E: S_E \times U \to \mathrm{Sym}^k(B),
  \end{equation*}
  where $S_E$ is an open neighborhood of $s_0$ in $S$. For more details on this, we refer the reader to \cite[Definition 4.3.1]{BarletMagnusson2019CyclesI}.

  An analytic family $(Z_s)_{s\in S}$ is called \textbf{proper} if the projection map $\pi_S: |\Gamma_S| \to S$ is proper. 
\end{rem}

If $X$ is K\"ahler, the following compactness property of the cycle space can be deduced from Bishop's compactness theorem \cite[Theorem 1]{Bishop1964}.

\begin{cor}[{\cite[Corollary 4.2.76]{BarletMagnusson2019CyclesI}}]\label{compactness property}
  Let $X$ be a compact K\"ahler manifold. Then the connected components of $\cC_n(X)$ are compact.
\end{cor}

\subsection{Proof of \autoref{rational sweepout}}

With these notions in place, we define the following subset of the Barlet space of $1$--cycles
\begin{equation*}
  \cR(X) \coloneq \big\{ C \, \big|\, C\, \text{is an irreducible reduced rational curve in }X \big\} \subset \cC_1(X).
\end{equation*}
As defined, $\cR(X)$ is in general not a closed analytic subspace of $\cC_1(X)$. However, a slightly weaker statement is sufficient for the proof of \autoref{rational sweepout}. Before we explain this in detail, we need the following definition.

\begin{definition}[{\cite[Definition 2.1]{Campana2004Appendix}}]
   We call a subset $A$ of a complex space $M$ \textbf{Zariski--regular} if for each irreducible closed analytic subset $N$ of $M$ the intersection $A \cap N$ either contains the general point of $N$ or is contained in a countable union of closed proper analytic subsets of $N$.
\end{definition}

\begin{prop}
  Suppose that $X$ is a compact complex analytic space. Then $\cR(X)$ is a Zariski--regular subset of $\cC_1(X)$. 
\end{prop}

\begin{proof}
  Suppose that $S \subset \cC_1(X)$ is an irreducible closed analytic subset. Recall that the weight of a compact cycle $C = \sum_{j=1}^r m_j C_j$ is $w(C) = \sum_{j=1}^r m_j$. It is a fact \cite[Proposition 4.7.2]{BarletMagnusson2019CyclesI} which follows from Remmert's direct image theorem applied to the proper holomorphic addition map of cycles, that the locus
  \begin{equation*}
      F_2 \coloneq \{ C \in \cC_1(X) \, | \, w(C)\geq 2\} 
  \end{equation*}
  is a closed analytic subset of the cycle space $\cC_1(X)$. Now, if $S \subset F_2$, then no general point of $S$ parametrizes an irreducible reduced curve. Consequently, $\cR(X) \cap S = \emptyset$. 

  On the other hand, if $S \nsubseteq F_2$, then $S \cap F_2$ is a proper closed analytic subset of $S$. Therefore the general point of $S$ parametrizes a reduced irreducible curve. Let us denote by $S^\circ$ this general part.

  Let us write $\pi: U_{S^\circ} \to S^\circ$ to denote the restriction of the tautological family of $\cC_1(X)$ to $S^\circ$. Let $\tilde{U}_{S^\circ}$ denote the normalization of the total space $U_S$. Now, on a subset $S^{\circ \circ}$ away from a proper closed analytic subset of $S^\circ$, the genus of the fiber of $\tilde{U}_{S^\circ}$ is constant. If this generic genus is $g=0$, then the intersection $S \cap \cR(X)$ contains the generic point of $S$. Else $S \cap \cR(X)$ is contained in the union of proper analytic subsets $S \setminus S^\circ$ and $S^\circ \setminus S^{\circ \circ}$. 
\end{proof}

The following is a consequence of the definition of Zariski--regularity and follows by an elementary induction argument. 

\begin{prop}[{\cite[Proposition 2.4]{Campana2004Appendix}}]\label{better covers}
  Let $M$ be a reduced complex space and suppose that $A \subset M$ is Zariski--regular. Then $A$ is contained in a countable union of irreducible closed analytic subsets $S_k$ such that $A \cap S_k$ contains the general point of $S_k$.
\end{prop}

After these preparations, we can now proceed with the proof of the main result of this section.

\begin{proof}[Proof of \autoref{rational sweepout}]
  Let us assume throughout that $X$ is connected.

  Since $\cR(X)$ is a Zariski--regular subset of $\cC_1(X)$, it is contained in a countable union of closed analytic subsets $S_k \subset \cC_1(X)$, $k \geq 1$ such that $\cR(X) \cap S_k$ contains the general point of $S_k$. We can also, without loss of generality, assume each $S_k$ to be contained in one connected component of $\cC_1(X)$. Moreover, since $X$ is K\"ahler, it follows from \autoref{compactness property} that each $S_k$ is compact, and consequently, the graphs $|\Gamma_{S_k}| \subset \cC_1(X) \times X$ are compact complex analytic subsets. 

  Further, observe that
  \begin{equation*}
    \Rat(X, I) \subset \bigcup_{k\geq 1} \pi_X (|\Gamma_{S_k}|),
  \end{equation*} 
  where $\pi_X$ is the projection of the graph $|\Gamma_{S_k}| \subset S_k \times X$ into $X$. Now, by Remmert's proper mapping theorem \cite[Chapter I, Theorem 8.8]{DemaillyCADG}, $Y_k \coloneq \pi_X (|\Gamma_{S_k}|)$ are closed analytic subsets of $X$. 

  It remains to notice that each $Y_k$ is a proper analytic subset of $X$. Indeed, if $Y_k = X$ for some $k$, then after normalizing the tautological family over $S_k$, the fact that the general point of $S_k$ is a rational curve implies that $X$ is uniruled, contradicting the hypothesis.
\end{proof}

\section{Fueter maps with incidence conditions}\label{Fueter maps with incidence conditions}

Let us begin by putting \autoref{thm:B} in context.

By \cite[Proof of Theorem 4.1]{HNS09}, the space
\begin{equation*}
  \cM^m \coloneq \{ (f, H) \in W^{k, p}(M, X) \times \cP^m  \,| \, \cF_H f = 0 \} 
\end{equation*}
is a $C^1$ Banach submanifold of $W^{k, p}(M, X) \times \cP^m$, where we have fixed once and for all $p$ and $k$ satisfying $k-3/p>2$. Here, the projection map $p^m : \cM^m \to \cP^m$ is Fredholm of index $0$. Consequently, from the Sard--Smale theorem \cite{Sm65} one can immediately deduce the following.

\begin{theorem}[{\cite[Theorem 4.1]{HNS09}}]
  Suppose that the data $(M, X, \cI)$ is as in \autoref{Setup}. For a generic $H \in \cP^m$, the solution space
  \begin{equation*}
    \cM(H) \coloneq \{ f \in W^{k, p}(M, X) \,| \, \cF_H f = 0 \} 
  \end{equation*}
  is a smooth $0$--dimensional submanifold of $W^{k, p}(M, X)$.
\end{theorem}

\autoref{thm:B} is a more refined version of this transversality theorem in the presence of incidence conditions. We begin with some preparatory lemmas. 

To ease the notation, for the remainder of this section, we set
\begin{equation*}
  \cB \coloneq W^{k, p}(M, X) \eand \cC_n \coloneq \Conf_n(M).
\end{equation*}

\begin{lemma}\label{incidence locus}
  The incidence locus 
  \begin{equation*}
    \cI_n \coloneq \{ (f, \bx) \in \cB \times \cC_n \, | \, f(x_j) \in Z \text{ and } T_{x_j}f(v) \in T_{f(x_j)}Z \text{ for all } j\}
  \end{equation*}
  is a $C^1$ Banach submanifold of $\cB \times \cC_n$ of codimension $4n$.
\end{lemma}

\begin{proof}
    Consider the evaluation map
  \begin{equation*}
    \bev=(\ev_1, \dots, \ev_n): \cB \times \cC_n \to X^n: \quad \bev(f, \bx) = \Big(f(x_1), \dots, f(x_n) \Big).
  \end{equation*}
  Clearly, this is a submersion: since $x_j$ are distinct, one can prescribe the values of a section $\xi \in W^{k,p} \Gamma(f^*TX)$ independently at all $x_j$. Consequently, 
  \begin{equation*}
    \cB^Z_n \coloneq \bev^{-1}(Z^n)
  \end{equation*}
  is a Banach submanifold of $\cB \times \cC_n$ of codimension $2n$.

  Now, the metric $g_X$ gives the canonical choice of the normal bundle $NZ$ of $Z$. Over $\cB^Z_n$ define a rank $2n$ vector bundle
  \begin{equation*}
    \cV_n \coloneq \bigoplus_{j=1}^n \ev_j^* NZ.
  \end{equation*}
  Write $\bx = (x_1, \dots, x_n)$. The vector field $v\in \Gamma(TM)$ gives rise to a section $\sigma_n \in C^1\Gamma(\cB^Z_n, \cV_n)$ defined by 
  \begin{equation*}
    \sigma_n(f, \bx) = \Big(\pr^{NZ}\big(T_{x_j}f(v)\big) \Big)_{j=1}^n,
  \end{equation*}
  where the $C^1$-regularity follows by the Sobolev embedding theorem from the assumption  that $k-3/p>2$. We claim that $\sigma_n$ is transverse to the zero section. Indeed, at a zero $(f, \bx)$, take arbitrary $\zeta_j \in N_{f(x_j)}Z$ for every $j=1, \dots, n$. Clearly, one can arrange $\eta \in W^{k, p}\Gamma(f^*TX)$ supported in disjoint neighborhoods of $x_j$ such that 
  \begin{align*}
    \eta(x_j)=0 \eand \big(\nabla_v^\perp \eta\big)_{x_j} = \zeta_j,
  \end{align*}
  so that for the linearization $D\sigma_n$ at $(f, \bx)$, taken using the normal connection induced by $g_X$, we have
  \begin{equation*}
    D\sigma_n(\eta, 0) = (\zeta_1, \dots, \zeta_n).
  \end{equation*}
  Consequently, the incidence locus $\cI_n = \sigma_n^{-1}(0)$ is a $C^1$ Banach submanifold of codimension $2n$ in $\cB_n^Z$ and thus a $C^1$ Banach submanifold of codimension $4n$ in $\cB \times \cC_n$.
\end{proof}

Write 
\begin{equation*}
  \cE'_f \coloneq W^{k-1,p}\Gamma(f^*TX).
\end{equation*}
For an $H$-Fueter map $f$, denote by 
\begin{equation*}
  \cD_{f, H}: T_f \cB \to \cE'_f
\end{equation*}
the linearization of the Fueter operator at $f$. With respect to the connection induced by $g_X$,
\begin{equation*}
  \cD_{f, H} \xi = \sum_{a=1}^3 \cI(e_a) \nabla_{e_a} \xi - \nabla_{\xi} \nabla^{g_X} H(\, \cdot \, , f (\,\cdot \,)).
\end{equation*}
This is a first-order elliptic operator, hence Fredholm of index $0$.

Henceforth, we will use the notation from the statement of \autoref{thm:B} and abbreviate $\cN^m_n (v, Z)$ by $\cN^m_n$. The following index computation will be useful in the proof of \autoref{thm:B}.

\begin{lemma}\label{index computation}
  For $(f, \bx, H) \in \cN^m_n$, the operator  
    \begin{equation*}
      \cD: T_{(f, \bx)} \cI_n \to \cE_{f}': \quad D(\xi, \eta) = \cD_{f, H} \xi
    \end{equation*}
  is Fredholm of index $-n$.
\end{lemma}

\begin{proof}
  Consider the operator
  \begin{equation*}
    \widetilde{\cD}: T_f \cB \oplus T_{\bx} \cC_n \to \cE_{f}': \; \widetilde{\cD}(\xi, \bm{\eta}) = \cD_{f, H} \xi,
  \end{equation*}
  which is clearly Fredholm of index
  \begin{equation*}
    \ind \widetilde{\cD} = \ind \cD_{f, H}  + \dim T_{\bx} \cC_n = 3n.
  \end{equation*}
  From \autoref{incidence locus}, we know that $T_{(f, \bx)}\cI_n \subset T_f \cB \oplus T_{\bx} \cC_n$ is a closed subspace of codimension $4n$. Hence the inclusion $\iota: T_{(f, \bx)}\cI_n \hookrightarrow T_f \cB \oplus T_{\bx} \cC_n$ is Fredholm of index $-4n$. Since $\cD = \widetilde{\cD} \circ \iota$, the composition formula for Fredholm operators \cite[Theorem 4.4.1]{BuehlerSalamon2018} implies that $\cD$ is Fredholm of index
  \begin{equation*}
    \ind \cD = \ind \widetilde{\cD} + \ind \iota = 3n-4n = -n. \qedhere
  \end{equation*}
\end{proof}

We now record a simple consequence of the freedom to vary the Hamiltonian.

\begin{lemma}\label{dense image}
  For every $f \in W^{k, p}(M, X)$, the operator 
  \begin{equation*}
    B_f: C^m(M \times X, \bR) \to W^{k-1, p}\Gamma(f^* TX): \; B_f(h)=\nabla^{g_X} h(\, \cdot \,, f(\, \cdot \,))
  \end{equation*}
  has dense image.
\end{lemma}

\begin{proof}
  Let $\iota: X \hookrightarrow \bR^N$ be an isometric embedding for some $N$ large enough. This embedding induces a universal projection map
  \begin{equation*}
    \Pi: X \to \End (\bR^N): \; \Pi(y) = \pr_{T_y \iota (T_y X)}
  \end{equation*}
  which is clearly smooth so that 
  \begin{equation*}
    \Pi_f \coloneq \Pi \circ f \in W^{k,p}(M, \End (\bR^N)).
  \end{equation*}
  If $s \in W^{k-1, p}\Gamma(f^*TX)$, then $f^*T\iota(s)\in W^{k-1, p}(M, \bR^N)$. This means that there is a sequence $(a_l)\in C^\infty(M, \bR^N)^\bN$ with 
  \begin{equation}\label{defining property of a_l}
    a_l  \to f^*T\iota(s) \quad \text{ in } W^{k-1, p}(M, \bR^N).
  \end{equation}
  Now, for every $l \in \bN$, define $h_l \in C^\infty(M\times X, \bR)$ by
  \begin{equation*}
    h_l(x, y) \coloneq \langle a_l(x), \iota(y) \rangle_{\bR^N_{\text{std}}}.
  \end{equation*}
  Since the embedding $\iota$ was chosen to be isometric, it holds that for all $x \in M$
  \begin{equation}\label{key property of h_l}
    T_{f(x)} \iota (\nabla^{g_X} h_l (x, f(x))) = \Pi_f (x) \cdot a_l(x).
  \end{equation} 
  Because $k-3/p>2$, the multiplication
  \begin{equation*}
    W^{k, p}(M, \End(\bR^N)) \times W^{k-1, p}(M, \bR^N) \to W^{k-1, p}(M, \bR^N)
  \end{equation*}
  is continuous, so that the combination of \eqref{defining property of a_l} and \eqref{key property of h_l}
  \begin{equation}\label{key convergence 1}
    f^*T \iota \Big( \nabla^{g_X} h_l( \, \cdot \, , f(\, \cdot \, ))\Big) \to f^*T \iota (s) \quad \text{ in } W^{k-1, p}(M, \bR^N).
  \end{equation}
  Since the bundle map 
  \begin{equation*}
    f^*T \iota: f^*TX \to \underline{\bR}^N
  \end{equation*}
  and its metric adjoint
  \begin{equation*}
    (f^*T \iota)^*: \underline{\bR}^N \to f^*TX
  \end{equation*} 
  have $W^{k,p}$ coefficients, by the Sobolev multiplication theorem they induce bounded maps on $W^{k-1,p}$. Now, as
  \begin{equation*}
    (f^*T \iota)^* (f^*T \iota )= \id_{f^*TX},
  \end{equation*} 
  we deduce from \eqref{key convergence 1} that 
  \begin{equation*}
    \nabla^{g_X} h_l( \, \cdot \, , f(\, \cdot \, )) \to s \quad \text{ in } W^{k-1, p}\Gamma(f^*TX),
  \end{equation*}
  proving the claim.
\end{proof}

We can now assemble the preceding lemmas to prove \autoref{thm:B}. The only remaining point is to apply the implicit function theorem to the universal Fueter section and compute the index of the projection.

\begin{proof}[Proof of \autoref{thm:B}]
  Let $\cE' \to \cB$ be a Banach bundle with fiber $\cE'_f = W^{k-1, p}\Gamma(f^*TX)$ and let
  \begin{equation*}
    \cE \to \cI_n \times \cP^m
  \end{equation*}
  be the pull back of $\cE'$ to $\cI_n \times \cP^m$ via the projection map. Consider the section $\fF_n \in C^1 \Gamma(\cE)$ defined by
  \begin{equation*}
    \cF_n(f, \bx, H) \coloneq \cF_H f,
  \end{equation*}
  where the $C^1$-regularity follows from the assumption $m\geq k+1$. Now observe that
  \begin{equation*}
    \cN^m_n = \fF^{-1}_n(0).
  \end{equation*}
  At a zero $(f, \bx, H)$, the linearization of $\fF_n$ with respect to the connection on $f^*TX$ induced by $g_X$ is
  \begin{equation*}
    D\fF_n: T_{(f, \bx)} \cI_n \oplus T_H \cP^m \to \cE'_f
  \end{equation*}
  given by
  \begin{equation*}
    D\fF_n(\xi, \bm{\eta}, h) = \cD(\xi, \bm{\eta}) - B_f(h),
  \end{equation*}
  where $\cD$ is the operator from \autoref{index computation} and $B_f$ is the operator from \autoref{dense image}.
  
  We first claim that $D\fF_n$ is surjective. Let 
  \begin{equation*}
    q: \cE'_f \to \coker \cD
  \end{equation*}
  denote the quotient map. By \autoref{dense image}, $B_f$ has dense image, and hence $q \circ B_f$ has dense image in $\coker \cD$. By  \autoref{index computation}, $\coker \cD$ is finite-dimensional, and therefore 
  \begin{equation*}
    q \circ B_f : T_H \cP^m \to \coker \cD
  \end{equation*}
  is surjective. Consequently
  \begin{equation*}
    \cE'_f = \im \cD + \im B_f = \im D\fF_n,
  \end{equation*}
  which proves surjectivity.

  It is not difficult to see that $D\fF_n$ has complemented kernel: since $\cD$ is Fredholm, choose a closed complement
  \begin{equation*}
    T_{(f, \bx)} \cI_n = \ker \cD \oplus V.
  \end{equation*}
  Since $q \circ B_f$ is surjective and $\coker \cD$ is finite-dimensional, we can choose a finite-dimensional subspace $W \subset T_H \cP^m$ such that
  \begin{equation*}
    q \circ B_f |_{W} : W \to \coker \cD
  \end{equation*}
  is an isomorphism. With this choice $D\fF_n|_{V \oplus W}: V \oplus W \to \cE_f'$ is an isomorphism, too, so that $D\fF_n$ admits a bounded right inverse. Finally, we use the implicit function theorem to deduce that $\cN^m_n$ is a $C^1$ Banach manifold. 

  It remains to compute the index of the projection $\pi^m: \cN^m_n \to \cP^m$. Applying the snake lemma to the commutative diagram 
  \[\begin{tikzcd}
  0 \arrow{r}{} &
  {T_{(f, \bx)}\cI_n}
    \arrow{r}{}
    \arrow{d}{\cD}
  &
  {T_{(f, \bx)}\cI_n \oplus T_{H} \cP^m}
    \arrow{r}{}
    \arrow{d}{{D\fF_n}}
  &
  {T_{H} \cP^m}
    \arrow{r}{}
    \arrow{d}{0}
  & 0 \\
  0 \arrow{r}{} &
  {\cE'_f}
    \arrow{r}{\mathrm{id}}
  &
  {\cE'_f}
    \arrow{r}{}
  &
  0 \arrow{r}{}
  & 0
  \end{tikzcd}\]
  yields
  \begin{equation*}
    0 \longrightarrow \ker \cD \longrightarrow \ker D\fF_n \longrightarrow T_H \cP^m \longrightarrow \coker \cD \longrightarrow \coker D\fF_n \longrightarrow 0.
  \end{equation*}
  Since $T_{(f, \bx, H)} \cN^m_n  = \ker D\fF_n$, this exact sequence can be identified with 
  \begin{equation*}
    0 \longrightarrow \ker \cD \longrightarrow T_{(f, \bx, H)} \cN^m_n \xrightarrow{T\pi^m} T_H \cP^m \longrightarrow \coker \cD \longrightarrow 0,
  \end{equation*}
  from which we immediately conclude that $T\pi^m$ is Fredholm with the kernel and the cokernel identified respectively with the kernel and the cokernel of $\cD$. In particular, \autoref{index computation} implies that
  \begin{equation*}
    \ind T\pi^m = \ind \cD = -n. \qedhere
  \end{equation*}
\end{proof}
\section{Fredholm rectifiability and generic avoidance}

The remaining ingredient in the proof of the main theorem is the generic avoidance of \autoref{thm:C}. Although this result is a rather straightforward application of standard arguments in transversality theory, we could not locate its statement in the literature.

The following lemma relates Fredholm rectifiable subsets of Banach spaces to small subsets of more general topological vector spaces. 

\begin{lemma}\label{Banach to Hausdorff}
  Let $V$ be a Hausdorff topological vector space, $W$ a Banach space and $j: V \to W$ a continuous linear map with dense image. Suppose that $B$ is a second-countable Banach manifold and $F: B \to W$ is a $C^1$-Fredholm map of negative index. Then $V \setminus j^{-1}(F(B))$ is comeager in $V$. 
\end{lemma}

\begin{proof}
  Pick a countable closed cover $\{ B_l \}$ of $B$ such that for all $l$ the map $F|_{B_l}$ is proper. This can be arranged since by an elementary application of the inverse function theorem, $F$ can be locally put in a form $F(x, k) = (x, f(x, k))$, where $k$ is in a finite-dimensional vector space; the properness follows by restricting $k$ to a compact subset. We will show that $A_l \coloneq j^{-1}(F(B_l))$ is nowhere dense. 

  Since $F|_{B_l}$ is proper, $F(B_l)$ is closed and so is $A_l$. It therefore remains to show that $A_l$ has empty interior. Suppose that $v_0 \in A_l$ and let $U$ be its open neighborhood. Set
  \begin{equation*}
    Z_0 \coloneq \{b \in B_l \, | \, F(b)=j(v_0) \}.
  \end{equation*}

  Note that since $Z_0$ is compact, $j(V)$ is dense and $F$ is a Fredholm map, there is an open neighborhood $O$ of $Z_0$ and a finite dimensional subspace $E \subset V$ such that 
  \begin{equation*}
    j(E) + \im T_b F = W \quad \text{for all } b \in O.
  \end{equation*}
  Now, consider
  \begin{equation*}
    \Psi: E \times O \to W: \quad \Psi(e, b) = F(b) - j(v_0+e).
  \end{equation*}
  This is a $C^1$ submersion by design, and hence the set 
  \begin{equation*}
    \cZ \coloneq \Psi^{-1}(0)
  \end{equation*}
  is a finite-dimensional $C^1$ manifold of dimension $\dim E + \ind TF < \dim E$. Consequently, by Sard's theorem the image of the projection 
  \begin{equation*}
    \pi: \cZ \to E: \quad \pi(e, b) = e
  \end{equation*}
  has empty interior. 

  Since $F|_{B_l}$ is proper, for all sufficiently small $e \in E$ (since $V$ is Hausdorff, the induced topology on $E$ is the one of a usual Euclidean space) every solution $b \in B_l$ of
  \begin{equation*}
    F(b) = j(v_0 +e)
  \end{equation*}
  must satisfy $b \in O$. Now, we can choose $e \in E$ arbitrarily small with $v_0 + e \in U$ such that $e \notin \im \pi$, which shows that $A_l$ has empty interior. Consequently, $j^{-1}(F(B)) = \bigcup_l A_l$ is meager.
\end{proof}

The following is a generic avoidance result in the Banach space context and is a straightforward consequence of the transversality theory of Fredholm maps as developed in \cite[Chapter 4]{AR67}.

\begin{prop}\label{Generic avoidance for Banach}
  Let $K$ be a compact smooth manifold of dimension $k$ and let $\cR \subset W$ be a codimension $n$ Fredholm $C^1$-rectifiable subset of a Banach space $W$. If $k < n$, then the exceptional locus
  \begin{equation*}
    \cE_K (\cR) \coloneq \{ f \in C^1(K, W) \, | \, f(K) \cap \cR \neq \emptyset \}
  \end{equation*}
  is Fredholm $C^1$ rectifiable of codimension $n-k$ in $C^1(K, W)$.
\end{prop}

\begin{proof}
  Write $\cR \subset \bigcup_i F_i (B_i)$, where $\ind TF_i \leq -n$. The evaluation map 
  \begin{equation*}
    \ev: K \times C^1(K, W) \to W: \; \ev(x, u) = u(x)
  \end{equation*}
  is a $C^1$ submersion. Consequently, the fiber product 
  \begin{equation*}
    \cI_i \coloneq \big(K \times C^1(K, W)\big) \times_W B_i
  \end{equation*}
  is a $C^1$ Banach manifold \cite[Corollary 17.2]{AR67}. The projection map 
  \begin{equation*}
    \pi_i: \cI_i \to C^1(K, W)
  \end{equation*}
  is Fredholm of index 
  \begin{equation*}
    \ind T \pi_i = \dim K + \ind TF_i \leq k-n,
  \end{equation*}
  which implies the result in view of 
  \begin{equation*}
    \cE_k(\cR) \subset \bigcup_i \pi_i (\cI_i). \qedhere
  \end{equation*}
\end{proof}

\begin{proof}[Proof of \autoref{thm:C}]
  This is an immediate consequence of \autoref{Banach to Hausdorff} and \autoref{Generic avoidance for Banach}.
\end{proof} 

\bibliographystyle{plain}
\bibliography{references}
\end{document}